\documentclass[12pt,final]{amsart}
\usepackage{amsmath,amsthm,amssymb}
\usepackage{amsfonts}
\usepackage[varg]{txfonts}
\usepackage[mathscr]{eucal}
\usepackage{bxpapersize}
\usepackage{ascmac}
\usepackage{epic, eepic}
\usepackage{graphicx}
\usepackage{indentfirst}
\usepackage{cases}
\usepackage{color}
\usepackage{url}
\usepackage{bm}
\usepackage[all]{xy}
\usepackage[top=25truemm,bottom=25truemm,left=30truemm,right=30truemm]{geometry}
\usepackage[noadjust]{cite}
\usepackage{showkeys}
\usepackage{enumerate}
\usepackage{hyperref}
\hypersetup{% hyperrefオプションリスト
setpagesize=false,
 bookmarksnumbered=true,%
 bookmarksopen=true,%
 colorlinks=true,%
 linkcolor=blue,
 citecolor=blue,
 urlcolor=blue,
}

\theoremstyle{plane}
\newtheorem{introthm}{Theorem}
\newtheorem{thm}{Theorem}[section]
\newtheorem{prop}[thm]{Proposition}%[section]
\newtheorem{lem}[thm]{Lemma}%[section]
\newtheorem{cor}[thm]{Corollary}%[section]
\newtheorem{fact}[thm]{Fact}%[section]
\theoremstyle{definition}
\newtheorem{dfn}[thm]{Definition}%[section]
\newtheorem{ex}[thm]{Example}%[section]

\theoremstyle{remark}
\newtheorem{rem}[thm]{Remark}%[section]
 \newtheorem*{acknowledgements}{Acknowledgements}

\newcommand{\rank}{\operatorname{rank}}

\newcommand{\im}{\operatorname{Im}}

\newcommand{\sgn}{\operatorname{sgn}}

\newcommand{\0}{\bm{0}}
\newcommand{\R}{\bm{R}}

\newcommand{\Sig}{\Sigma}
\renewcommand{\phi}{\varphi}
\newcommand{\eps}{\varepsilon}

\newcommand{\wtil}{\widetilde}
\newcommand{\inner}[2]{\left\langle{#1},{#2}\right\rangle}

\numberwithin{equation}{section}

\begin{document}
%%%%%TEXT START%%%%%

\title[Gaussian curvature and Gauss map]
{On Gaussian curvatures and singularities of Gauss maps of cuspidal edges}
\author[K. Teramoto]{Keisuke Teramoto}
\thanks{This work was partially supported by JSPS KAKENHI Grant Number 19K14533.}
\address{Institute of Mathematics for Industry, %Graduate School of Science, 
Kyushu University, 
744 Motooka, Fukuoka 819-0395, Japan}
\email{k-teramoto@imi.kyushu-u.ac.jp}

\date{\today}

\maketitle

\begin{abstract}
We show relation between sign of Gaussian curvature of cuspidal edge and geometric invariants through 
types of singularities of Gauss map. 
Moreover, we define and characterize positivity/negativity of cusps of Gauss maps by geometric invariants of cuspidal edges, 
and show relation between sign of cusps and of the Gaussian curvature. 
\end{abstract}

%%%%%Section 1%%%%%
\section{Introduction} 
Let $f\colon\Sig\to\R^3$ be a $C^\infty$ map, where $\R^3$ is the Euclidean $3$-space and $\Sig$ is a domain of $\R^2$. 
Then a point $p\in\Sig$ is said to be a {\it singular point} of $f$ if $\rank df_p<2$ holds. 
We denote by $S(f)=\{q\in \Sig\ |\ \rank df_q<2\}$ the set of singular points of $f$. 
A singular point $p\in S(f)$ of $f$ is a {\it cuspidal edge} 
if there exist local diffeomorphisms $\phi\colon\Sig\to\R^2$ on the source and 
$\Phi\colon\R^3\to\R^3$ on the target such that $\Phi\circ f\circ\phi^{-1}(u,v)=(u,v^2,v^3)$, 
where $u,v$ are coordinates of $\R^2$, namely, $f$ is {$\mathcal{A}$-equivalent} to the germ $(u,v)\mapsto(u,v^2,v^3)$ at $\0$. 
(In general, two map germs $f,g\colon(\R^n,\0)\to(\R^m,\0)$ are {\it$\mathcal{A}$-equivalent} if there exist diffeomorphism germs 
$\phi\colon(\R^n,\0)\to(\R^n,\0)$ on the source and $\Phi\colon(\R^m,\0)\to(\R^m,\0)$ on the target such that $\Phi\circ f\circ\phi^{-1}=g$ holds.) 

If $f$ at $p$ is a cuspidal edge, then $\rank df_p=1$ holds, that is, $p$ is a {\it $($co$)$rank one singularity} of $f$. 
It is known that a cuspidal edge is a fundamental singularity of a {front} in $3$-space (see \cite{agv,ifrt}). 
Here, a $C^\infty$ map $f\colon\Sig\to\R^3$ is said to be a {\it front} 
if there exists a $C^\infty$ map $\nu\colon\Sig\to S^2$ such that 
\begin{itemize}
\item $\inner{df_q(X)}{\nu(q)}=0$ for any $q\in \Sig$ and $X\in T_q\R^2$\quad (orthogonality condition), 
\item $(f,\nu)\colon \Sig\to\R^3\times S^2$ gives an immersion\quad (immersion condition),
\end{itemize}
where $S^2$ is the unit sphere in $\R^3$ and $\inner{\cdot}{\cdot}$ is the canonical inner product of $\R^3$. 
We call $\nu$ the {\it Gauss map} of $f$. 
By definition, fronts admit certain singularities and the Gauss map even at singular points, and hence they might be considered as a generalization of immersions. 
There are several studies of surfaces with singularities such as fronts (or {\it frontals} which satisfy the above orthogonality condition) from the differential geometric viewpoint 
(cf. \cite{fsuy,fukutaka,hhnsuy,hks,hnsuy,hnuy,im,is,istaka,istake2,krsuy,ms,mst,msuy,mu,st,ot,suy,suy2,suy1,suy3,t1,t2,t3,t4}). 

We assume that $f$ at $p$ is a cuspidal edge in the following.  
Then there exist a neighborhood $U$ of $p$ and a regular $C^\infty$ curve $\gamma=\gamma(t)\colon(-\eps,\eps)\to U$ with $\gamma(0)=p$ 
such that $\im(\gamma)=S(f)\cap U$, 
where $\im(\gamma)$ is the image of $\gamma$. 
We remark that $\gamma$ consists of corank one singularities of $f$. 
Since $\rank df_p=1$, there exists a non-zero vector field $\eta$ on $U$ such that 
$df_q(\eta_q)=\0$ for any $q\in S(f)\cap U$. 
%and $\det(\gamma'(t),\eta(t))|_{t=0}\neq0$ holds, where $\gamma'(t)=(d\gamma/dt)(t)$ and $\eta(t)=\eta(\gamma(t))$ (\cite{krsuy}).
We call $\gamma$ and $\eta$ a {\it singular curve} and a {\it null vector field}, respectively (cf. \cite{krsuy,suy,suy1}). 
%We say that such a singular point is {\it non-degenerate singular point} of a front. 
%We note that a null vector can be taken at a corank one singular point of a $C^\infty$ map (see \cite{istake2,s1,suy1}). 
%We note that $\det(\gamma'(t),\eta(t))|_{t=0}\neq0$ holds, where $\gamma'(t)=(d\gamma/dt)(t)$ and $\eta(t)=\eta(\gamma(t))$. 

We set two functions $\lambda,\Lambda\colon U\to\R$ on $U$ by 
\begin{equation}\label{eq:lambdas}
\lambda(u,v)=\det(f_u,f_v,\nu)(u,v),\quad \Lambda(u,v)=\det(\nu_u,\nu_v,\nu)(u,v)
\end{equation}
for some coordinates $(u,v)$ on $U$, where $(\ )_u=\partial/\partial u$ and $(\ )_v=\partial/\partial v$. 
We call $\lambda$ and $\Lambda$ the {\it signed area density function} of $f$ and the {\it discriminant function} of $\nu$, respectively (cf. \cite{suy,suy1}). 
By definition, $S(f)=\lambda^{-1}(0)$ holds, in particular $\lambda(\gamma(t))=0$. 
The following useful criterion for a cuspidal edge using $\eta$ and $\lambda$ is known (\cite{krsuy,suy1}). 
\begin{fact}\label{fact:crit-ce}
Let $f\colon\Sig\to\R^3$ be a front and $p\in S(f)$ a corank one singular point of $f$. 
Then $p$ is a cuspidal edge of $f$ if and only if $\eta\lambda(p)\neq0$, 
where $\eta\lambda$ is a directional derivative of $\lambda$ in the direction $\eta$. 
%Moreover, $\eta\lambda(p)\neq0$ is equivalent to $\det(\gamma'(t),\eta(t))|_{t=0}\neq0$. 
\end{fact}
We remark that useful criteria for other corank one singularities of fronts and frontals are known (cf. \cite{fsuy,hks,is,istaka,krsuy,suy1}). 

Using $\lambda$ and $\Lambda$, the Gaussian curvature $K$ of $f$ is given as $K=\Lambda/\lambda$ on $U\setminus S(f)$ by the Weingarten formula. 
In general, $K$ is unbounded near $S(f)$ since $\lambda=0$ on $S(f)$. 
However, using a geometric invariants $\kappa_\nu$ called the {\it limiting normal curvature} (\cite{msuy,suy}), the following assertion holds. 
\begin{fact}[\cite{msuy,suy1,suy3}]\label{fact:Gauss-kn}
The Gaussian curvature $K$ is bounded near a cuspidal edge $p\in S(f)\cap U$ 
if and only if $\kappa_\nu$ vanishes along $\gamma$ $($see Figure \ref{fig:kappan}$)$. 
\end{fact}
\begin{figure}[htbp]
  \begin{center}
    \begin{tabular}{c}

      % 1
      \begin{minipage}{0.33\hsize}
        \begin{center}
          \includegraphics[width=3cm]{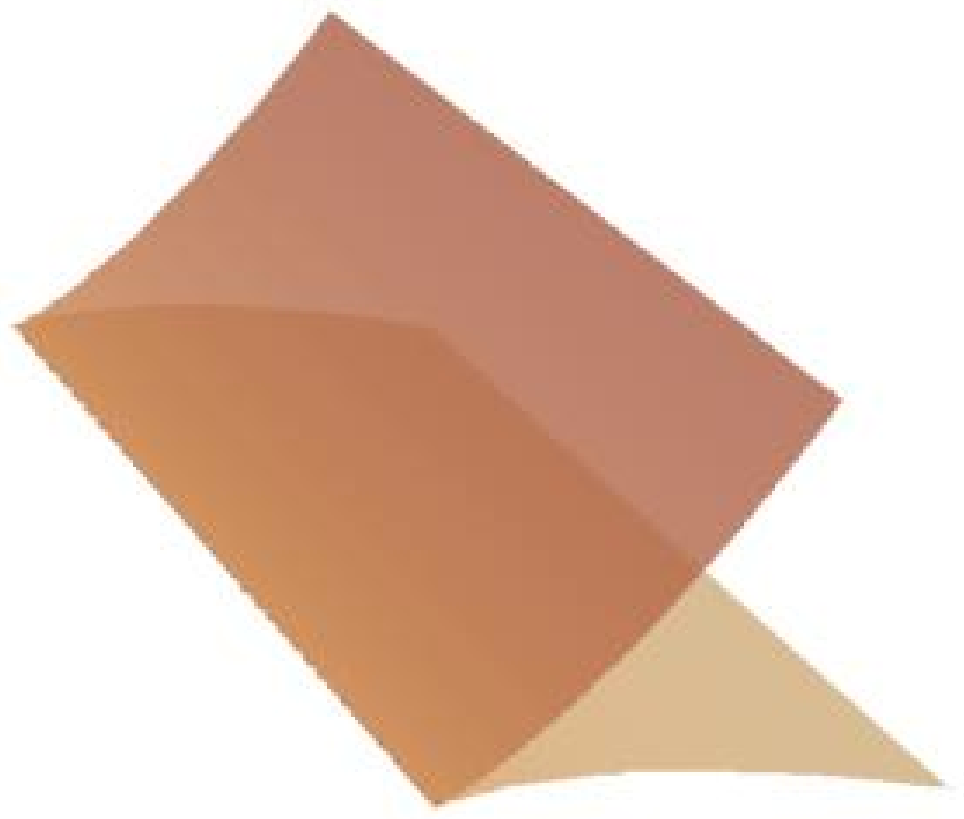}
         % \hspace{1cm} 
        \end{center}
      \end{minipage}

      % 2
      \begin{minipage}{0.33\hsize}
        \begin{center}
          \includegraphics[width=3cm]{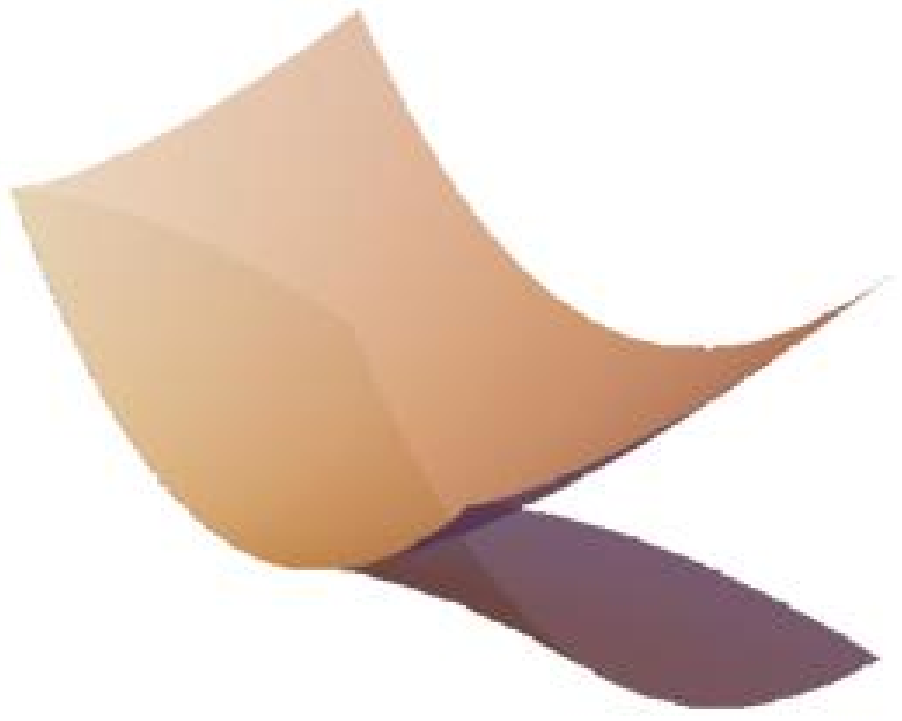}
          %\hspace{1cm} 
        \end{center}
      \end{minipage}

    \end{tabular}
    \caption{Cuspidal edges with vanishing $\kappa_\nu$ (left) and non-vanishing $\kappa_\nu$ (right).}
    \label{fig:kappan}
  \end{center}
\end{figure}
On the other hand, the mean curvature $H$ of a front diverges near a corank one singular point (\cite{mst,msuy,suy}). 
We note that if the Gaussian curvature $K$ of a front $f$ is bounded near a cuspidal edge $p$, 
$\Lambda(\gamma(t))=0$ holds, 
where $\gamma$ is a singular curve of $f$ through $p$ and $\Lambda$ is the function as in \eqref{eq:lambdas}. 
This implies that $\im(\gamma)$ is a subset of the singular set $S(\nu)=\Lambda^{-1}(0)$ of $\nu$. 
Thus the Gauss map $\nu$ has singularities in such cases. 
It is known that a {\it fold} and a {\it cusp} 
%which are $\mathcal{A}$-equivalent to the germs $(u,v)\mapsto(u,v^2)$ and $(u,v)\mapsto(u,v^3+u v)$, respectively, 
naturally appear as  singularities of the Gauss map (cf. \cite{bgm,bl,w}). 
If $\nu$ at $p$ is a fold or a cusp, then there exists a $C^\infty$ regular curve $\sigma(\tau)$ ($\tau<|\delta|$) such that 
$\sigma$ parametrizes $S(\nu)$ near $p$ in general (see \cite{ifrt,w}). 
We call $\sigma$ a {\it parabolic curve} of $f$, and 
such singular points are called {\it non-degenerate singular points} of $\nu$.

In the case that the Gaussian curvature $K$ is bounded, the following assertion about shapes of a cuspidal edge is known. 
\begin{fact}[{\cite[Theorem 3.1]{suy}}]\label{fact:Gauss-ks}
Let $f\colon\Sig\to\R^3$ be a front with a cuspidal edge $p$. 
Suppose that the Gaussian curvature $K$ of $f$ is bounded sufficiently small neighborhood $U$ of $p$. 
If $K$ is positive $($resp. non-negative$)$ on $U$, 
then the singular curvature $\kappa_s$ is negative $($resp. non-positive$)$ at $p$ $($\/see Figure \ref{fig:kappas}\/$)$.
\end{fact}
\begin{figure}[htbp]
  \begin{center}
    \begin{tabular}{c}

      % 1
      \begin{minipage}{0.33\hsize}
        \begin{center}
          \includegraphics[width=3.5cm]{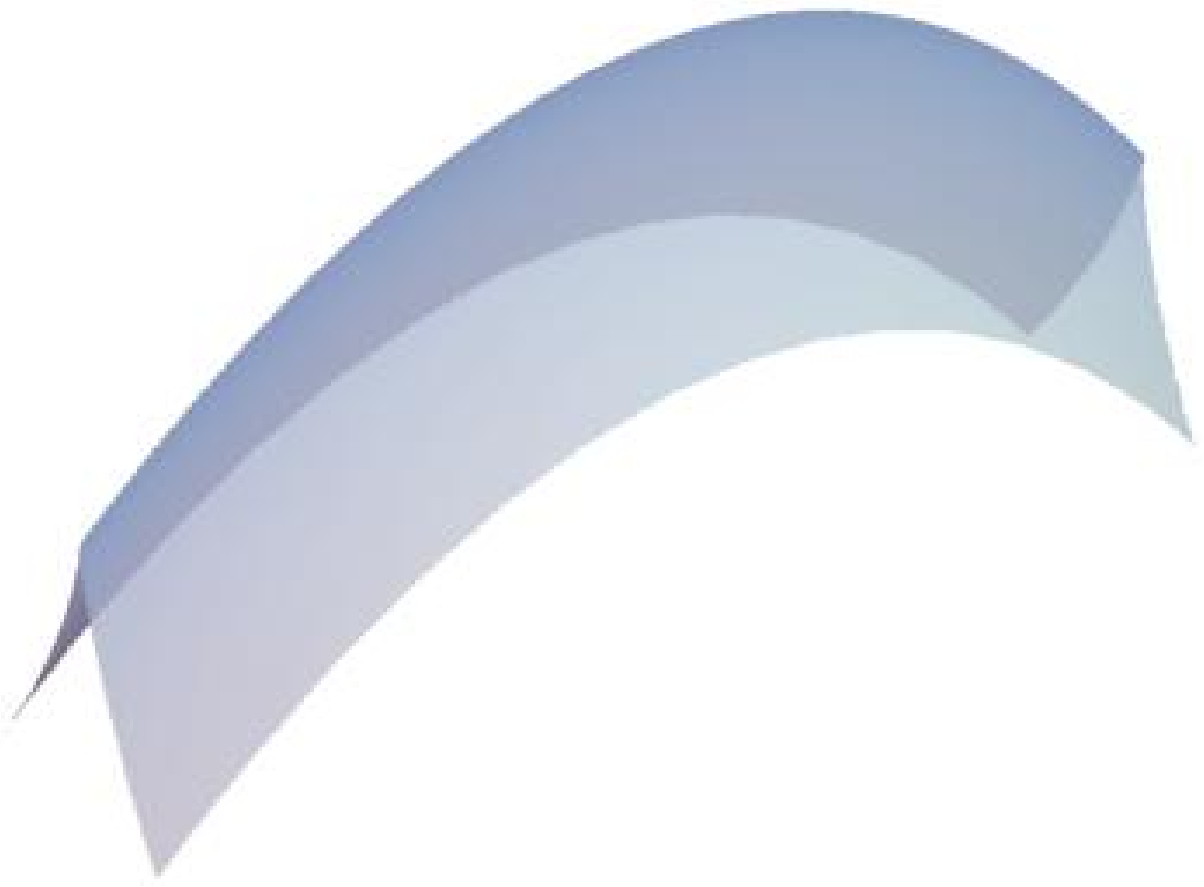}
         % \hspace{1cm} 
        \end{center}
      \end{minipage}

      % 2
      \begin{minipage}{0.33\hsize}
        \begin{center}
          \includegraphics[width=3.5cm]{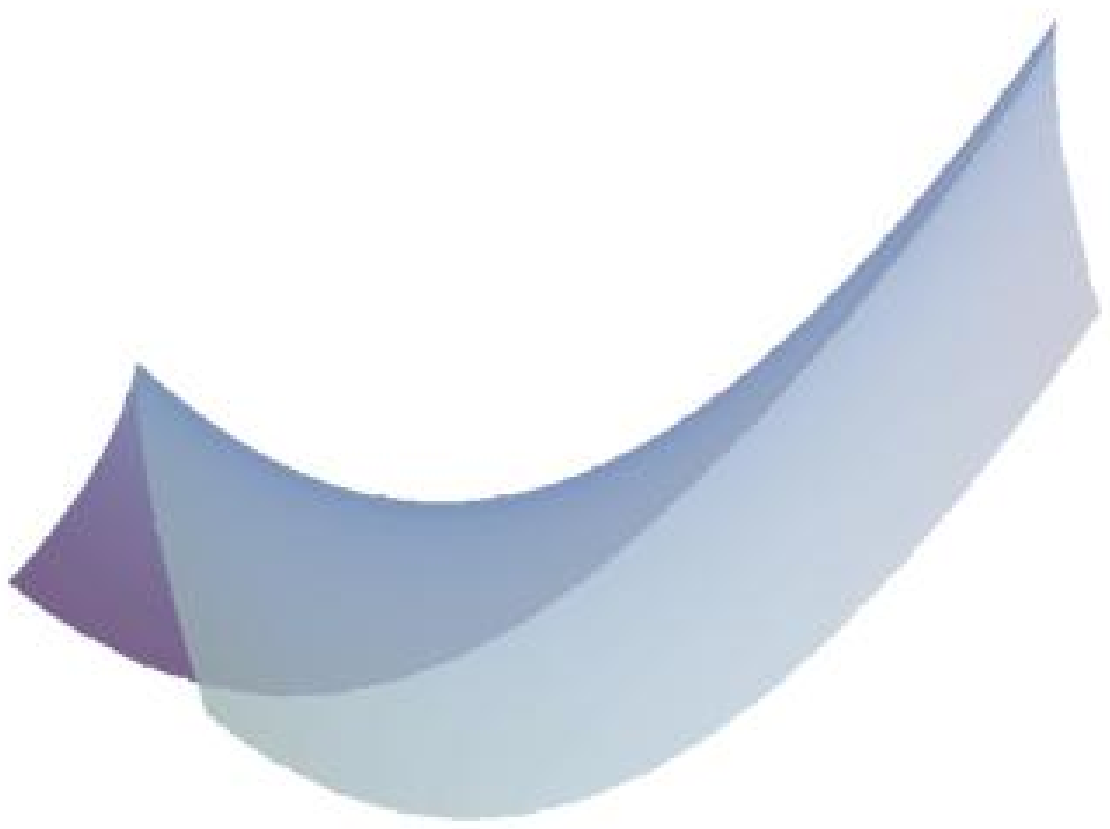}
          %\hspace{1cm} 
        \end{center}
      \end{minipage}

    \end{tabular}
    \caption{Cuspidal edges with positive $\kappa_s$ (left) and negative $\kappa_s$ (right).}
    \label{fig:kappas}
  \end{center}
\end{figure}
Here the singular curvature $\kappa_s$ is an {\it intrinsic invariant} of a cuspidal edge (\cite{hhnsuy,hnsuy,msuy,suy}). 
This statement tells us that if $K$ is positive and bounded, then a cuspidal edge is curved concavely. 
We should mention that the inverse of statement as in Fact \ref{fact:Gauss-ks} is {\it not} true in general (cf. \cite{suy}).
%\begin{equation}\label{eq:statement}
%\textrm{If $K$ is bounded and negative near a cuspidal edge, then $\kappa_s$ is positive.}
%\end{equation}
Thus it is natural to ask the following question: {\it When does the inverse statement as in Fact \ref{fact:Gauss-ks} hold?}

In this paper, we shall give an answer to this question. 
More precisely, we show the following. 
\begin{introthm}\label{thm:Gauss-ks}
Let $f\colon\Sig\to\R^3$ be a front with a cuspidal edge $p$ and $\nu$ its Gauss map. 
Suppose that the Gaussian curvature $K$ of $f$ is bounded on a sufficiently small neighborhood $U$ of $p$. 
When $p$ is a non-degenerate singular point of $\nu$ but not a fold, 
then $K$ is positive $($\/resp. negative$\/)$ on $U$ if and only if $\kappa_s$ is negative $($\/resp. positive$\/)$ at $p$.
\end{introthm}

We next focus on the singular locus $\check{\nu}=\nu\circ\sigma$ of the Gauss map $\nu$, where $\sigma$ is a parabolic curve.
The singular locus of a fold is a regular spherical curve 
and of a cusp is a spherical curve with an {\it ordinary cusp singularity} which is $\mathcal{A}$-equivalent to $t\mapsto(t^2,t^3)$. 
For the case of a cusp singularity of the Gauss map, one can define the {\it cuspidal curvature} $\mu^\nu$ 
for $\check{\nu}$ (\cite{suy2,shibaume}). 
Using the cuspidal curvature $\mu^\nu$, we can define {\it positivity} and {\it negativity} for a cusp singularity 
(cf. \eqref{eq:cuspcurvature1} and Definition \ref{def:zigzag-Gauss}). 
If $K$ is non-zero bounded near a cuspidal edge, then one can take $\sigma(t)$ as $\sigma(t)=\gamma(\pm t)$ near a cuspidal edge (see \cite{suy3}). 
In particular, we shall show the following. 
\begin{introthm}\label{cor:sign}
Let $f\colon\Sig\to\R^3$ be a front, $\nu\colon\Sig\to S^2$ the Gauss map of $f$, 
$p\in\Sig$ a cuspidal edge of $f$ and $\gamma(t)$ $(\/t<|\eps|\/)$ 
the singular curve of $f$ passing through $p(=\gamma(0))$. 
Take an orientation of $\gamma$ so that the left-hand side of $\gamma$ is $\lambda>0$ on a neighborhood $U(\subset\Sig)$ of $p$, 
where $\lambda$ is the signed area density function of $f$ as in \eqref{eq:lambdas}. 
Suppose that the Gaussian curvature $K$ of $f$ is non-zero bounded on $U$ and $\nu$ has a cusp at $p$. 
Then $p$ is a positive $($\/resp. negative\/$)$ cusp of the singular locus $\check{\nu}=\nu\circ\sigma$ of $\nu$ 
if and only if $\kappa_s$ is positive $($\/resp. negative$)$ at $p$ when we take 
the parabolic curve $\sigma(t)$ through $p=\sigma(0)$ as $\sigma(t)=\gamma(t)$.
\end{introthm}

%%%%%Section 2%%%%%
\section{Geometric properties of cuspidal edges}
Let $f\colon\Sig\to\R^3$ be a front with a cuspidal edge at $p\in\Sig$, where $\Sig$ is a domain of $\R^2$. 
Then one can take the following local coordinate system. 
\begin{dfn}[\cite{krsuy,suy1,msuy}]\label{def:adapted}
A local coordinate system $(U;u,v)$ centered at a cuspidal edge $p$ 
is {\it adapted} if it is compatible with respect to the orientation of $\Sig$ 
and the following conditions hold:
\begin{itemize}
\item the $u$-axis gives a singular curve, that is, $\gamma(u)=(u,0)$ on $U$,
\item $\eta=\partial_v$ is a null vector field,
\item there are no singular points other than the $u$-axis.
\end{itemize}
Moreover, we call a local coordinate system $(U;u,v)$ a {\it special adapted} 
if it is an adapted coordinate system and the pair $\{f_u,f_{vv},\nu\}$ is an orthonormal frame along the $u$-axis.
\end{dfn}
%
%This local coordinate system is based on the criterion for a cuspidal edge (see Fact \ref{fact:crit-ce}). 
In what follows, we fix an adapted coordinate system $(U;u,v)$ centered at a cuspidal edge $p$ of a front $f$ in $\R^3$. 
On this local coordinate system, we set the following invariants along the $u$-axis:
\begin{align}\label{kappas}
\begin{aligned}
\kappa_s&=\sgn(\lambda_v)\dfrac{\det(f_u,f_{uu},\nu)}{|f_u|^3},\quad
\kappa_\nu=\dfrac{\inner{f_{uu}}{\nu}}{|f_u|^2},\quad
\kappa_c=\dfrac{|f_u|^{3/2}\det(f_u,f_{vv},f_{vvv})}{|f_u\times f_{vv}|^{5/2}},\\
\kappa_t&=\dfrac{\det(f_u,f_{vv},f_{uvv})}{|f_u\times f_{vv}|^2}
-\dfrac{\inner{f_u}{f_{vv}}\det(f_u,f_{vv},f_{uu})}{|f_u|^2|f_u\times f_{vv}|^2}.%\\
%\kappa_i&=\dfrac{\det(f_u,f_{vv},f_{uuu})}{|f_u|^3|f_u\times f_{vv}|}-
%3\dfrac{\inner{f_u}{f_{uu}}\det(f_u,f_{vv},f_{uu})}{|f_u|^5|f_u\times f_{vv}|},
\end{aligned}
\end{align}
These invariants $\kappa_s$, $\kappa_\nu$, $\kappa_c$ and $\kappa_t$ %and $\kappa_i$ 
are called the {\it singular curvature} (\cite{suy}), the {\it limiting normal curvature} (\cite{suy}) 
and the {\it cuspidal curvature} (\cite{msuy}), the {\it cuspidal torsion} (\cite{ms}), respectively. 
We note that $\kappa_c$ does not vanish along the $u$-axis when all singular points consist of cuspidal edges (\cite[Proposition 3.11]{msuy}). 
For $\kappa_t$, the following assertion is known. 
\begin{fact}[\cite{istake2,t2}]\label{fact:curve-line}
Let $f\colon\Sig\to\R^3$ be a front with a cuspidal edge $p$ and $\gamma$ 
a singular curve passing through $p$. 
Then the singular locus $\hat{\gamma}=f\circ\gamma$ $($or the singular curve $\gamma\/)$ 
is a line of curvature of $f$ if and only if $\kappa_t$ vanishes identically 
along the singular curve $\gamma$.
\end{fact}
For other geometric properties of these invariants, see \cite{hhnsuy,hnuy,hnsuy,msuy,ms,suy,istake2,t2,t1,t3} for example. 

On the other hand, we can take a $C^\infty$ map $h\colon U\to\R^3\setminus\{\0\}$ 
satisfying $f_v=vh$ because $f_v(u,0)=\0$ and $\eta\lambda(u,0)=\lambda_v(u,0)=\det(f_u,f_{vv},\nu)(u,0)\neq0$. 
Thus $\{f_u,h,\nu\}$ forms a frame. 
Using these maps, we define the following functions on $U$: 
$\wtil{E}=|f_u|^2$, $\wtil{F}=\inner{f_u}{h}$, $\wtil{G}=|h|^2$, 
$\wtil{L}=-\inner{f_u}{\nu_u}$, $\wtil{M}=-\inner{h}{\nu_u}$ and 
$\wtil{N}=-\inner{h}{\nu_v}$, 
where $|\cdot|=\sqrt{\inner{\cdot}{\cdot}}$.
If we take a special adapted coordinate system $(u,v)$, then invariants as in \eqref{kappas} can be written as 
\begin{equation}\label{eq:curvature1}
\kappa_s=-\dfrac{\wtil{E}_{vv}}{2},\quad \kappa_\nu=\wtil{L},\quad \kappa_c=2{\wtil{N}},\quad \kappa_t=\wtil{M}
\end{equation}
along the $u$-axis (see \cite[Proposition 1.8]{hhnsuy}, \cite[Lemma 3.4]{hnsuy} and \cite[Lemma 2.7]{t2}). 
\begin{lem}\label{lem:Weingarten}
Take a special adapted coordinate system $(U;u,v)$ centered at a cuspidal edge $p$. 
Then the differentials $\nu_u$ and $\nu_v$ of the Gauss map $\nu$ can be expressed as 
\begin{equation}\label{eq:diff-nu}
\nu_u=-\kappa_\nu f_u-\kappa_t h,\quad 
\nu_v=-\dfrac{\kappa_c}{2}h
\end{equation}
along the $u$-axis.
\end{lem}
\begin{proof}
By \cite[Lemma 2.1]{t1}, it holds that  
\begin{equation*}
\nu_u=\frac{\wtil{F}\wtil{M}-\wtil{G}\wtil{L}}{\wtil{E}\wtil{G}-\wtil{F}^2}f_u+
\frac{\wtil{F}\wtil{L}-\wtil{E}\wtil{M}}{\wtil{E}\wtil{G}-\wtil{F}^2}h,\ 
\nu_v=\frac{\wtil{F}\wtil{N}-v\wtil{G}\wtil{M}}{\wtil{E}\wtil{G}-\wtil{F}^2}f_u+
\frac{v\wtil{F}\wtil{M}-\wtil{E}\wtil{N}}{\wtil{E}\wtil{G}-\wtil{F}^2}h.
\end{equation*}
Since $\wtil{E}(u,0)=\wtil{G}(u,0)=1$, $\wtil{F}(u,0)=0$ and \eqref{eq:curvature1} hold, we have the assertion.
\end{proof}

We now recall principal curvatures. 
Let $f$ be a front with cuspidal edge $p$. 
Take an adapted coordinate system $(U;u,v)$ around $p$. 
Then we set functions $\kappa_j\colon U\setminus\{v=0\}\to\R$ ($j=1,2$) as 
$$\kappa_1=H+\sqrt{H^2-K},\quad \kappa_2=H-\sqrt{H^2-K},$$
where $K$ and $H$ are the Gaussian and the mean curvature defined on $U\setminus\{v=0\}$. 
These functions are {\it principal curvatures} of $f$ since $\kappa_1\kappa_2=K$ and $2H=\kappa_1+\kappa_2$ hold. 
It is known that one of $\kappa_j$ ($j=1,2$) can be extended as a bounded $C^\infty$ function on $U$, say $\kappa$, 
and another, write $\tilde{\kappa}$, diverges near the $u$-axis (\cite{mu,t1,t2}). 
Moreover, $\kappa=\kappa_\nu$ holds along the $u$-axis, and $\hat{\kappa}:=\lambda\tilde{\kappa}$ 
is a bounded $C^\infty$ function on $U$ and proportional to $\kappa_c$ on the $u$-axis (cf. \cite[Theorem 3.1 and Remark 3.2]{t2}). 
In particular, it holds that
\begin{equation}\label{eq:k-hat}
\hat{\kappa}(p)=\dfrac{\kappa_c(p)}{2}
\end{equation}
at a cuspidal edge $p$ when we take a special adapted coordinate system $(U;u,v)$ around $p$ (see \cite[Remark 3.2]{t2} and \eqref{eq:curvature1}). 
%%%%%Section 3%%%%%
\section{Singularities of Gauss maps and the Gaussian curvature}
We consider the Gauss map $\nu$ of a front $f\colon\Sig\to\R^3$ with a cuspidal edge $p$. 
Let $(U;u,v)$ be an adapted coordinate system around $p$. 
Then $\kappa$ and $\tilde{\kappa}$ denote the bounded $C^\infty$ principal curvature and the unbounded principal curvature on $U$, respectively. 
Let $\Lambda\colon U\to\R$ be the discriminant function of $\nu$ as in \eqref{eq:lambdas}. 
Then the set of singular points of $\nu$ is $S(\nu)=\{q\in U\ |\ \Lambda(q)=0\}$. 
By the Weingarten formula (cf. \cite[Lemma 2.1]{t1}), we have 
$$\Lambda(u,v)=\kappa(u,v)\hat{\kappa}(u,v)=\lambda(u,v)K(u,v)=\hat{K}(u,v),$$
where $\lambda$ is the signed area density function of $f$ as in \eqref{eq:lambdas} and $\hat{\kappa}=\lambda\tilde{\kappa}$. 
Since $\hat{\kappa}(p)\neq0$, $\Lambda(p)=0$ if and only if $\kappa(p)=\kappa_\nu(p)=0$ (\cite{msuy,t4}). 
Thus we may assume that $S(\nu)=\{q\in U\ |\ \kappa(q)=0\}$ holds locally. 
By Fact \ref{fact:Gauss-kn}, if the Gaussian curvature $K$ is bounded near $p$, 
then the $u$-axis is also a set of singular points of $\nu$. 
We say that a singular point $q\in S(\nu)$ of $\nu$ is {\it non-degenerate} if $(\partial_u\kappa(q),\partial_v\kappa(q))\neq(0,0)$, 
where $\partial_u\kappa=\partial \kappa/\partial u$ and $\partial_v\kappa=\partial \kappa/\partial v$. 
Otherwise, we say a {\it degenerate singular point} of $\nu$. 
If $p$ is a non-degenerate singular point of $\nu$, 
then it follows from the implicit function theorem that 
there exist a neighborhood $V$ of $p$ and a regular curve $\sigma=\sigma(\tau)\colon(-\delta,\delta)\to V$ ($\delta>0$) such that 
$\sigma(0)=p$ and $\Lambda(\sigma(\tau))=0$ on $V$ (cf. \cite{bgm,bl,ifrt,s1,w}). 
We call the curve $\sigma$ the {\it parabolic curve} of $f$ or the {\it singular curve} of $\nu$.
%Using geometric invariants, we can characterize non-degeneracy as follows. 
%
\begin{lem}[{cf. \cite[Lemma 3.5]{t4}}]\label{lem:non-degeneracy}
Let $f\colon\Sig\to\R^3$ be a front with a cuspidal edge $p$ and 
$\nu$ the Gauss map of $f$. % defined on a neighborhood $U$ of $p$. 
Suppose that $p$ is also a singular point of $\nu$. 
Then $p$ is a non-degenerate singular point of $\nu$ if and only if 
$\kappa_\nu'(p)\neq0$ or $4\kappa_t(p)^2+\kappa_s(p)\kappa_c(p)^2\neq0$ holds.
\end{lem}
\begin{proof}
Let us take a special adapted coordinate system $(U;u,v)$ centered at $p$. 
Let $\kappa$ be the bounded $C^\infty$ principal curvature of $f$ on $U$. 
Then $\partial_u\kappa(p)=\kappa_\nu'(p)$ holds since $\kappa(u,0)=\kappa_\nu(u)$. 
On the other hand, we have 
\begin{equation}\label{eq:kv}
\partial_v\kappa(p)=-\dfrac{4\kappa_t(p)^2+\kappa_s(p)\kappa_c(p)^2}{2\kappa_c(p)}\neq0
\end{equation}
(see \cite[Proposition 2.8]{t3}). 
Thus the assertion holds.
\end{proof}
%By Fact \ref{fact:Gauss-kn}, when the Gaussian curvature is bounded near a cuspidal edge $p$, 
%$p$ is also a singular point of $\nu$ and $\kappa_\nu=0$ along $\gamma$. 
%Thus the following assertion related to the Gaussian curvature holds.
%
\begin{prop}\label{thm:K-nu}
Let $f\colon\Sig\to\R^3$ be a front with a cuspidal edge $p$. 
Suppose that the Gaussian curvature $K$ is bounded on a neighborhood $U$ of $p$. 
Then $p$ is a non-degenerate singular point of the Gauss map $\nu$ 
if and only if $K$ takes non-zero value near $p$. 
\end{prop}
\begin{proof}
Let $(U;u,v)$ be a special adapted coordinate system around $p$. 
Suppose that $K$ is bounded on $U$. 
Then by Fact \ref{fact:Gauss-kn}, $\kappa(u,0)=\kappa_\nu(u)=0$. 
Thus there exists a function $\psi\colon U\to\R$ such that $\kappa(u,v)=v\psi(u,v)$ by the division lemma (\cite{gg}), 
and hence $\psi(p)=\partial_v\kappa(p)$. 
The Gaussian curvature $K$ is given as 
$$K=\kappa\tilde{\kappa}=\dfrac{\kappa\hat{\kappa}}{\lambda}=\dfrac{\psi\hat{\kappa}}{\det(f_u,h,\nu)}$$
on $U$, where $\hat{\kappa}=\lambda\tilde{\kappa}$ and $\lambda=v\det(f_u,h,\nu)$. 
Since $\psi(p)=\partial_v\kappa(p)$ (see \eqref{eq:kv}), $2\hat{\kappa}(p)=\kappa_c(p)$ (see \eqref{eq:k-hat}) and $\det(f_u,h,\nu)(p)=1$, we have
\begin{equation}\label{eq:Gaussian-curvature}
4K(p)=-4\kappa_t(p)^2-\kappa_s(p)\kappa_c(p)^2
\end{equation}
(cf. \cite[Remark 3.19]{msuy}). 
By Lemma \ref{lem:non-degeneracy}, we get the conclusion.
\end{proof}
\begin{rem}\label{rem:parabolic}
By this proposition, if the Gaussian curvature $K$ of a front $f$ is bounded at cuapidal edge $p$ and the corresponding Gauss map 
has a non-degenerate singularity at $p$, then $K$ is automatically non-zero near $p$. 
Moreover, if $K$ of $f$ is non-zero bounded, $S(f)=S(\nu)$ locally. 
This means that the parabolic curve $\sigma$ of $f$ satisfies $\sigma(t)=\gamma(\pm t)$ for sufficient small $0<t<|\eps|$.
\end{rem}

\begin{rem}
Proposition \ref{thm:K-nu} corresponds to the following statement: 
a cuspidal edge $p$ of a front $f$ is also a non-degenerate singular point of $\nu$ 
if and only if $\log|K|$ does not vanish on $U\setminus S(f)$, 
where $U$ is a neighborhood of $p$. 
This can be found in \cite[Lemma 3.25]{suy3} in more general situation. 
Thus Proposition \ref{thm:K-nu} might be considered as a rephrasing of \cite[Lemma 3.25]{suy3} in terms of 
singularities of the Gauss map.
\end{rem}

We next consider types of singularities of the Gauss map. 
It is known that generic singularities of the Gauss map are a {\it fold} and a {\it cusp}, 
which are $\mathcal{A}$-equivalent to the germs  $(u,v)\mapsto(u,v^2)$ and $(u,v)\mapsto(u,v^3+uv)$ at $\0$, respectively (cf. \cite{bgm,bl,po,w}). 
These singularities are non-degenerate singularities of the Gauss map. 
For singular points of a $C^\infty$ map between $2$-dimensional manifolds, see \cite{bgm,bl,r,s1,w}.
%For a cuspidal edge with bounded Gaussian curvature, we have the following.
%
\begin{fact}[{\cite[Proposition 3.12]{t4}}]\label{fact:singgauss}
Let $f\colon\Sig\to\R^3$ be a front with a cuspidal edge $p\in \Sig$ and $\nu$ its Gauss map. 
If the Gaussian curvature $K$ of $f$ is non-zero bounded on a sufficiently small neighborhood of $p$, 
then 
\begin{itemize}
\item $\nu$ at $p$ is a fold if and only if $\kappa_t(p)(4\kappa_t(p)^2+\kappa_s(p)\kappa_c(p)^2)\neq0$,
\item $\nu$ at $p$ is a cusp if and only if $\kappa_t(p)=0$, $\kappa_t'(p)\neq0$ and $\kappa_s(p)\neq0$.
\end{itemize}
\end{fact}
By Lemma \ref{lem:non-degeneracy} and Fact \ref{fact:singgauss}, we have the following assertion immediately.
\begin{cor}\label{cor:not-fold}
Let $f$ be a front in $\R^3$ with a cuspidal edge $p$. 
Suppose that the Gaussian curvature $K$ of $f$ is non-zero bounded near $p$. 
Then $p$ is a non-degenerate singular point other than a fold of $\nu$ if and only if $\kappa_t(p)=0$ and $\kappa_s(p)\neq0$.
\end{cor}
If $\kappa_t$ vanishes identically along the singular curve, 
namely, if the singular curve is a line of curvature (cf. Fact \ref{fact:curve-line}), the next assertion holds.
\begin{prop}\label{prop:cone-like}
Let $f\colon\Sig\to\R^3$ be a front and $p\in\Sig$ a cuspidal edge. 
Suppose that the Gaussian curvature $K$ of $f$ is non-zero bounded on a neighborhood $U$ of $p$, 
and the singular curve $\gamma(t)$ $($in $U)$ through $p$ is a line of curvature of $f$. 
Then the singular locus $\check{\nu}(t)=\nu(\sigma(t))$ of the Gauss map $\nu$ is a single point in $S^2$, 
where $\sigma(t)=\gamma(\pm t)$.
\end{prop}
\begin{proof}
Let us take a special adapted coordinate system $(U;u,v)$ centered at $p$.
Assume that $K$ is non-zero bounded on $U$. 
Then by Proposition \ref{thm:K-nu}, $p$ is also a non-degenerate singular point of $\nu$. 
One may choose the parabolic curve $\sigma$ satisfying $\sigma(u)=\gamma(u)$. 
We consider the behavior of $\nu\circ\sigma(u)=\nu\circ\gamma(u)=\nu(u,0)$. 
Since $\kappa_\nu=0$, we have $\nu_u=- \kappa_t(u)h(u,0)$ by \eqref{eq:diff-nu} in Lemma \ref{lem:Weingarten}. 
Thus $\nu_u$ is a zero vector along the $u$-axis by the assumption, and hence 
$\nu$ is a constant map along the $u$-axis. 
For the case of $\sigma(u)=\gamma(-u)$, we can show in a similar way. 
\end{proof}
This says that the singular locus $\nu\circ\gamma$ of the Gauss map $\nu$ is a single point 
if $K$ is non-zero bounded near a cuspidal edge $p$ and $\gamma$ is a line of curvature (see Remark \ref{rem:parabolic}).
In particular, if a surface of revolution has a cuspidal edge and non-zero bounded Gaussian curvature, 
the image of its Gauss map degenerates to a point. 
\begin{ex}\label{ex:rotation}
Let $f\colon(-2\pi,2\pi)\times(0,2\pi)\ni(u,v)\mapsto f(u,v)\in\R^3$ be a map given by 
$$f(u,v)=((2+\cos{u})\cos{v},(2+\cos{u})\sin{v},u-\sin{u}).$$
This is a cycloid of revolution whose singularities are cuspidal edges. 
By direct calculations, one can see that $S(f)=\{u=0\}$, $\eta=-\partial_u$ and the Gaussian curvature $K$ is negative bounded. 
Actually, $K=-1/4(2+\cos{u})<0$ holds. 
Moreover, $\kappa_s>0$ and $\kappa_t\equiv0$ hold along the $v$-axis. 
On the other hand, the Gauss map $\nu$ of $f$ is 
$$\nu(u,v)=\left(\sin{\dfrac{u}{2}}\cos{v},\sin{\dfrac{u}{2}}\sin{v},\cos{\dfrac{u}{2}}\right).$$
Note that $S(\nu)=S(f)=\{u=0\}$. 
By Proposition \ref{prop:cone-like}, $\nu(S(\nu))$ degenerates to a point. 
In fact, $\nu(0,v)=(0,0,1)$ holds (see Figure \ref{fig:cycloid}). 
\begin{figure}[htbp]
  \begin{center}
    \begin{tabular}{c}

      % 1
      \begin{minipage}{0.33\hsize}
        \begin{center}
          \includegraphics[width=3cm]{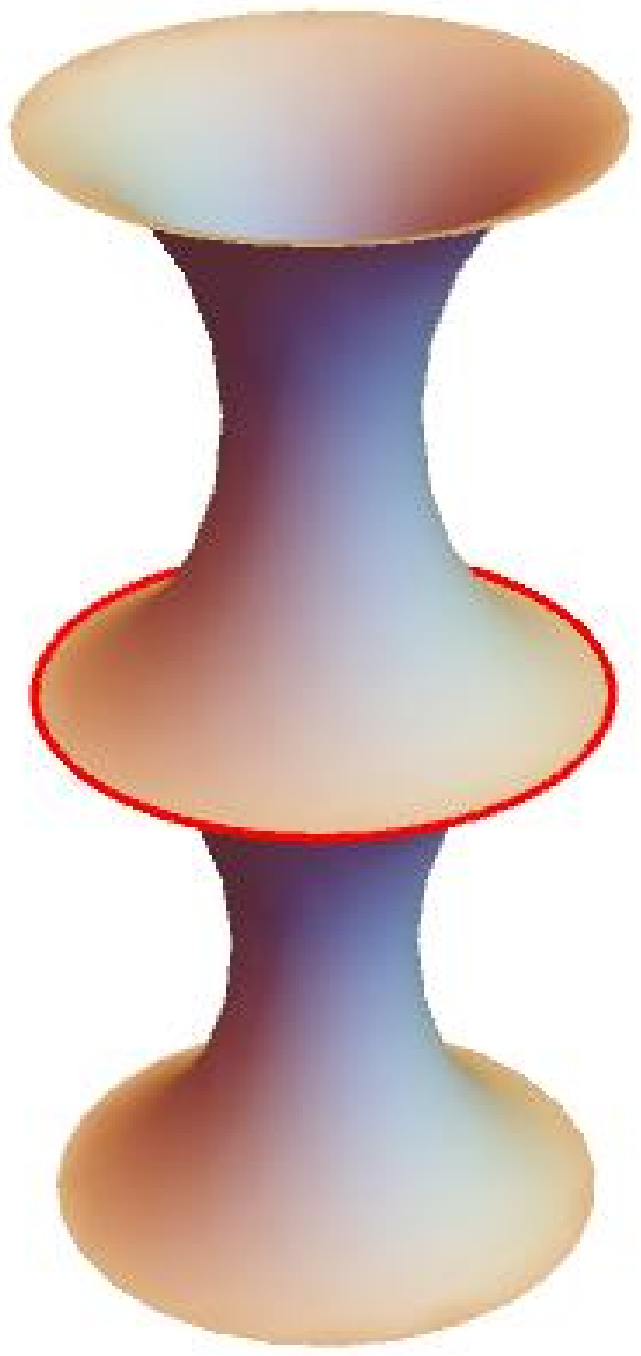}
         % \hspace{1cm} 
        \end{center}
      \end{minipage}

      % 2
      \begin{minipage}{0.33\hsize}
        \begin{center}
          \includegraphics[width=4cm]{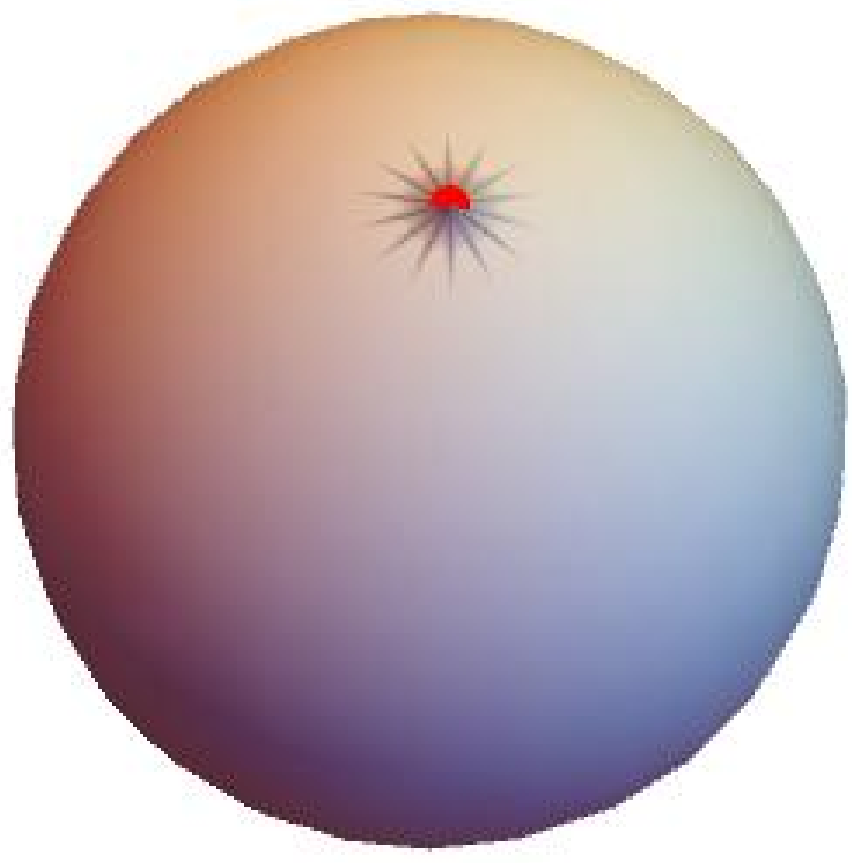}
          %\hspace{1cm} 
        \end{center}
      \end{minipage}

    \end{tabular}
    \caption{Cycloid of revolution (left) and its Gauss map image (right) of Example \ref{ex:rotation}. 
    A curve in the left figure and a point in the right figure which are colored by red are the singular loci of $f$ and $\nu$, respectively.}
    \label{fig:cycloid}
  \end{center}
\end{figure}
\end{ex}
We remark that a similar statement holds between 
a flat front in the hyperbolic $3$-space $H^3$ and its $\Delta_1$-dual flat front in the de Sitter $3$-space $S^3_1$ (see \cite[Corollary 4.3]{st}).

By using above results, we give a proof of Theorem \ref{thm:Gauss-ks}.
\begin{proof}[Proof of Theorem \ref{thm:Gauss-ks}]
Let $(U;u,v)$ be a special adapted coordinate system centered at a cuspidal edge $p$. 
Since $p$ is a non-degenerate singular point of the Gauss map $\nu$ and the Gaussian curvature $K$ is bounded, 
$K$ takes non-zero value on $U$. 
Since $\nu$ at $p$ is not a fold, $\kappa_t(p)=0$ and $\kappa_s(p)\neq0$ by Corollary \ref{cor:not-fold}. 
Thus $K$ can be written as 
$$4K(p)=-\kappa_s(p)\kappa_c(p)^2$$
by \eqref{eq:Gaussian-curvature} as in the proof of Proposition \ref{thm:K-nu}. 
This completes the proof.
\end{proof}

%%%%%Section 4%%%%%
\section{Signs of cusps of the Gauss map of a cuspidal edge}
We consider the sign of a cusp of the Gauss map of a front with a cuspidal edge. 
First, we recall the cuspidal curvature for a curve with an ordinary cusp singularity. 

Let $S^2$ be the standard unit sphere in the Euclidean $3$-space $\R^3$.
Note that for any $\bm{x}\in S^2$, 
the tangent plane $T_{\bm{x}}S^2$ of $S^2$ at $\bm{x}$ can be regarded as 
%the perpendicular component 
$\bm{x}^{\perp}=\{\bm{v}\in\R^3\ |\ \inner{\bm{v}}{\bm{x}}=0\}$ of $\bm{x}\in S^2\subset\R^3$ 
by identifying $T_{\bm{x}}\R^3$ with $\R^3$ and considering $T_{\bm{x}}S^2\subset\R^3$ as a vector subspace. 
%Setting $\Delta=\{(\bm{x},\bm{y})\in S^2\times S^2\ |\ \inner{\bm{x}}{\bm{y}}=0\}$, we note that $\Delta$ is a contact $3$-manifold.
%We define the set $\Delta\subset S^2\times S^2$ by 
%$$\Delta=\{(\bm{x},\bm{y})\in S^2\times S^2\ |\ \inner{\bm{x}}{\bm{y}}=0\},$$
%where $\inner{\cdot}{\cdot}$ is a canonical inner product of $\R^3$. 
%The set $\Delta$ is a $3$-dimensional contact manifold with the contact form $\theta_1=\inner{d\bm{x}}{\bm{y}}=\sum_{i=1}^3y_idx_i=0$. 

Let $c\colon I\to S^2$ be a $C^\infty$ map, where $I\subset\R$ is an open interval with a local coordinate $t$. 
Then we call the curve $c$ a {\it spherical curve}. 
Suppose that a point $t_0\in I$ is a {singular point} of $c$, that is, 
$(dc/dt)(t_0)=\dot{c}(t_0)=\0$ holds. 
It is known that $c$ has an {ordinary cusp} at $t_0$ if and only if 
$$%\Omega(D_t\dot{\sigma},D_tD_t\dot{\sigma})(t_0)=
\det(c,D_t\dot{c},D_tD_t\dot{c})(t_0)\neq0,$$ 
where %$\Omega$ is a Riemannian volume element on $S^2$, 
$D$ is the covariant derivative of $S^2$ with $D_t=D_{d/dt}$ (cf. \cite{suy1}).

Let $c\colon I\to S^2$ be a spherical curve with an ordinary cusp at $t_0\in I$. 
Then we set 
\begin{equation}\label{eq:cuspcurvature1}
\mu%=\left.\dfrac{\Omega(D_t\dot{\sigma}(t),D_tD_t\dot{\sigma}(t))}{|D_t\dot{\sigma}(t)|^{5/2}}\right|_{t=t_0}
=\left.\dfrac{\det(c(t),D_t\dot{c}(t),D_tD_t\dot{c}(t))}{|D_t\dot{c}(t)|^{5/2}}\right|_{t=t_0}.
\end{equation}
%where $D$ is a covariant derivative of $S^2$ (induced from $\R^3$\/) and 
%$\det$ is the determinant of $3\times3$ matrices.
This is a geometric invariant called a {\it cuspidal curvature} of $c$ (\cite{suy2,shibaume}). 
The cuspidal curvature does not depend on orientation preserving diffeomorphisms on the source and isometries on the target.
We note that the cuspidal curvature $\mu$ can be defined for curves with ordinary cusps in any Riemannian $2$-manifold (\cite{suy2}). 

%\begin{dfn}[cf. \cite{suy,suy2,shibaume}]\label{dfn:posicusp}
Let $c\colon I\to S^2$ be a spherical curve with 
an ordinary cusp at $t_0\in I$. 
Then $t_0$ is a {\it positive cusp} or a {\it zig} (resp. a {\it negative cusp} or a {\it zag}) of $c$ 
if the cuspidal curvature $\mu$ of $c$ is positive (resp. negative) at $t_0$ (cf. \cite{suy,suy2,shibaume}) (see Figure \ref{fig:cusp}).
%\end{dfn}
\begin{figure}[htbp]
  \begin{center}
    \begin{tabular}{c}

      % 1
      \begin{minipage}{0.33\hsize}
        \begin{center}
          \includegraphics[width=3.5cm]{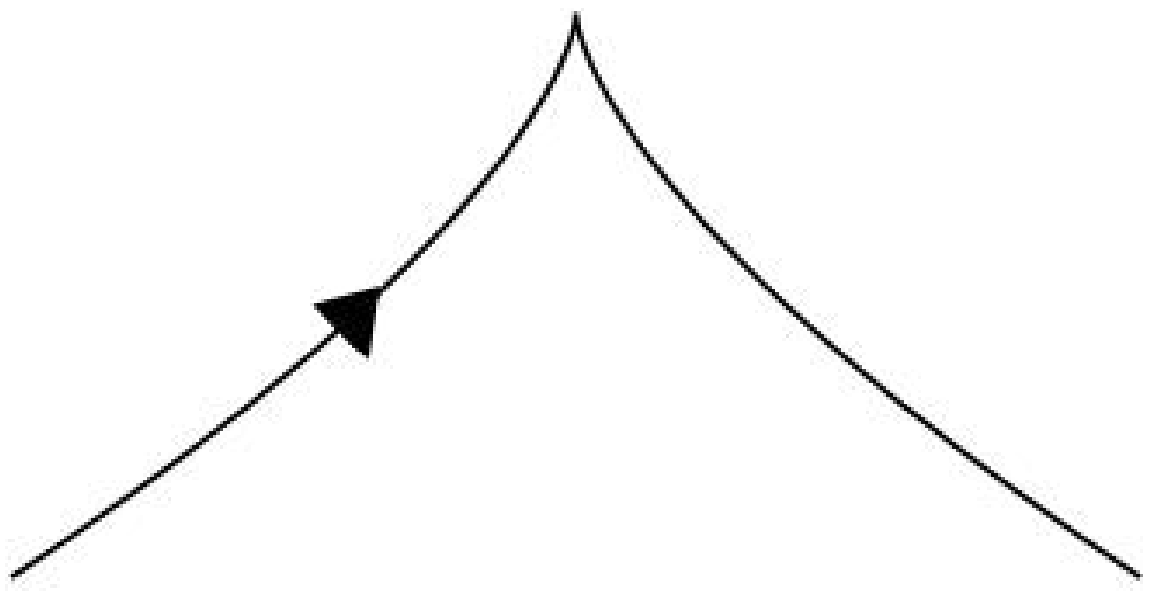}
         % \hspace{1cm} 
        \end{center}
      \end{minipage}

      % 2
      \begin{minipage}{0.33\hsize}
        \begin{center}
          \includegraphics[width=3.5cm]{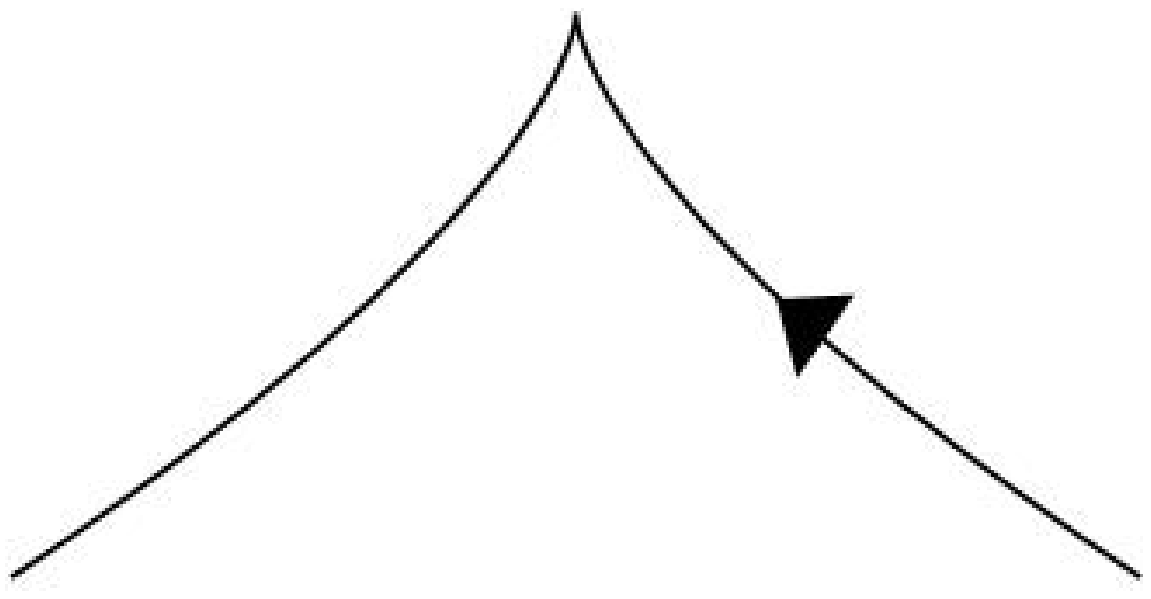}
          %\hspace{1cm} 
        \end{center}
      \end{minipage}

    \end{tabular}
    \caption{Left: positive cusp (zig). Right: negative cusp (zag). 
    A curve turns right (resp. left) for a cusp 
    if the cuspidal curvature $\mu$ is positive (resp. negative) at that point.}
    \label{fig:cusp}
  \end{center}
\end{figure}
We remark that global studies of {\it zigzag numbers} are known (see \cite{arnold,fukutaka0,suy,suy1}).

Let $f\colon\Sig\to\R^3$ be a front with a cuspidal edge $p\in\Sig$. 
Assume that the Gauss map $\nu$ of $f$ has a cusp at $p$ and the Gaussian curvature $K$ of $f$ is non-zero bounded on a neighborhood $U$ of $p$. 
Then there exists a parabolic curve $\sigma(t)$ passing through $p=\sigma(0)$.
In this case, the singular locus $\check{\nu}(t):=\nu\circ\sigma(t)(=\nu\circ\gamma(\pm t))$ is a spherical curve with an ordinary cusp singularity at $t=0$. 
{Moreover, we take an orientation of $\gamma(t)$ so that the left-hand side of $\gamma$ is the region of $\lambda>0$ on $U$.}
%Thus we can calculate the singular curvature $\mu^\nu$ for $\check{\nu}$ at $t=0$. 

\begin{dfn}\label{def:zigzag-Gauss}
Under the above setting, 
we call a point $p$ a {\it positive cusp} or a {\it zig} (resp. a {\it negative cusp} or a {\it zag}) of $\check{\nu}=\nu\circ\sigma$ 
if we take $\sigma(t)=\gamma(t)$ and the cuspidal curvature $\mu^\nu$ of the singular locus $\check{\nu}$ of $\nu$ 
is positive (resp. negative) at $p$, 
where $\gamma$ and $\sigma$ are a singular curve and a parabolic curve of $f$ through $p$, respectively. 
%Moreover, we say that $p$ is a {\it zig} (resp. a {\it zag}) of $\check{\nu}$ if $\mu^\nu>0$ (resp. $\mu^\nu<0$) at $p$. 
\end{dfn}
%We note that $\gamma'(0)\cdot\dot{\sigma}(0)=0$ does not hold (see \cite{t4}).

\begin{lem}\label{lem:orientation}
%Let $f\colon\Sig\to\R^3$ be a front with a cuspidal edge $p$. 
%Let $\nu$ be the Gauss map of $f$ and $\gamma(t)$ be a singular curve passing through $p$. 
%Suppose that the Gaussian curvature $K$ of $f$ is non-zero bounded on a sufficiently small neighborhood $U$ of $p$ 
%and the Gauss map $\nu$ has a cusp at $p$. 
%Take an orientation of $\gamma$ so that the left-hand side of $\gamma$ is a region of $\lambda>0$ on $U$, 
%where $\lambda$ is the signed area density function as in \eqref{eq:lambdas}. 
Under the above situation, 
although we change the Gauss map $\nu$ to $-\nu$ of a front $f$ with cuspidal edge $p$ whose Gaussian curvature is non-zero bounded near $p$, 
the positivity or negativity of a cusp of the singular locus of the Gauss map does not change. 
\end{lem}
\begin{proof}
By the definition of the cuspidal curvature \eqref{eq:cuspcurvature1}, if we change $\nu$ to $-\nu$, then the cuspidal curvature changes its sign. 
Moreover, if  we change the orientation of $\sigma(=\gamma)$, then the cuspidal curvature also changes the sign by \eqref{eq:cuspcurvature1}. 
On the other hand, if we change $\nu$ to $-\nu$, $\lambda$ changes to $-\lambda$. 
Thus we need to change the orientation of $\sigma=\gamma$ so that the region of $-\lambda>0$ is the left-hand side of $\gamma$ (cf. Definition \ref{def:zigzag-Gauss}). 
In this case, the positivity or negativity of a cusp does not change by the above discussion. 
\end{proof}

%To prove Theorem \ref{cor:sign}, we show the following lemma. 
\begin{proof}[Proof of Theorem \ref{cor:sign}]
Let $(U;u,v)$ be a special adapted coordinate system centered at $p$ 
and assume that the Gaussian curvature $K$ is non-zero bounded on $U$. 
{Then it holds that $\lambda=v\det(f_u,h,\nu)>0$ on the left-hand side of the $\gamma(u)=(u,0)$, 
the limiting normal curvature $\kappa_\nu$ vanishes identically along the $u$-axis (cf. Fact \ref{fact:Gauss-kn})
and the $u$-axis is also the set of singular points of $\nu$, that is, $\sigma(u)=(u,0)(=\gamma(u))$ on $U$ (cf. Proposition \ref{thm:K-nu} and \cite[Lemma 3.25]{suy3}), 
where $h\colon U\to\R^3\setminus\{\0\}$ is a $C^\infty$ map satisfying $f_v=vh$.} 
Thus the singular locus $\check{\nu}$ of $\nu$ is given as $\check{\nu}(u)=\nu(u,0)$. 
Moreover, by \eqref{eq:diff-nu} in Lemma \ref{lem:Weingarten} and the assumption, we have 
$$\nu_u=-\kappa_t h,\quad \nu_v=-\dfrac{\kappa_c}{2}h$$
along the $u$-axis. 
Thus $\check{\nu}'(u)=-\kappa_t(u)h(u,0)$ ($'=d/du$). 
%Since $p$ is a Whitney cusp of $\nu$, $\kappa_t(p)=0$, namely, $\check{\nu}'(p)(=\nu_u(p))=0$, 
%where $'=d/du$ (Fact \ref{fact:singgauss} \ref{Kbound}). 

To calculate the cuspidal curvature $\mu^\nu$ of $\check{\nu}$, 
we compute the second and the third order differentials of $\check{\nu}$. 
By direct calculations, we see that 
\begin{equation*}
\check{\nu}''=\dfrac{d^2\check{\nu}}{du^2}=-\kappa_t'h-\kappa_th_u,\quad
\check{\nu}'''=\dfrac{d^3\check{\nu}}{du^3}=-\kappa_t''h-2\kappa_t'h_u-\kappa_th_{uu}.
\end{equation*}
Since $\{f_u,h,\check{\nu}\}$ is an orthonormal frame along the $u$-axis, 
$h_u$ can be written as a linear combination of $f_u$, $h$ an $\check{\nu}$. 
We set 
$$h_u(u,0)=X_1(u)f_u(u,0)+X_2(u)h(u,0)+X_3(u)\check{\nu}(u)$$
along the $u$-axis, where $X_i$ $(i=1,2,3)$ are $C^\infty$ functions.
Since $\inner{h}{h}=1$ on the $u$-axis, $X_2=0$. 
Moreover, $X_3=\wtil{M}=\kappa_t$ hold because $\inner{h}{\check{\nu}}=0$ and 
$\inner{h_u}{\check{\nu}}+\inner{h}{\check{\nu}'}=0$. 
We consider $X_1$. 
Since $f_v=vh$, we see that $h_u=f_{uvv}$ on the $u$-axis. 
On the other hand, $\wtil{E}_{vv}=\inner{f_u}{f_u}_{vv}=2\inner{f_u}{f_{uvv}}=2\inner{f_u}{h_u}$ holds 
along the $u$-axis. 
Thus by \eqref{eq:curvature1}, we have $X_1=-\kappa_s$, 
and hence $h_u(u,0)$ is written as 
$$h_u(u,0)=-\kappa_s(u)f_u(u,0)+\kappa_t(u)\check{\nu}(u).$$
Since $\kappa_t(p)=0$ and $\kappa_t'(p)\neq0$ by Fact \ref{fact:singgauss}, we have $\check{\nu}'(p)=\bm{0}$, 
$$D_u\check{\nu}'(p)=-\kappa_t'(p)h(p),\quad 
D_uD_u\check{\nu}'(p)=-\kappa_t''(p)h(p)+2\kappa_t'(p)\kappa_s(p)f_u(p).$$
Therefore it holds that 
$$\det(\check{\nu},D_u\check{\nu}',D_uD_u\check{\nu}')=2\kappa_s(\kappa_t')^2,\quad
|D_u\check{\nu}'|^{5/2}=(\kappa_t')^2\sqrt{|\kappa_t'|}$$
at $p$. 
Thus the cuspidal curvature $\mu^{\nu}$ of $\check{\nu}$ at $p$ is 
$$\mu^\nu=\dfrac{2\kappa_s(p)}{\sqrt{|\kappa_t'(p)|}},$$
and hence we have the assertion. 
If we choose $-\nu$ as the Gauss map of $f$, then we have the same conclusion by Lemma \ref{lem:orientation}. 
\end{proof}

%\begin{rem}
%In the proof of Theorem \ref{cor:sign}, %\Lemma \ref{lem:sign}, 
%if we take the parabolic curve $\sigma$ of $f$ as $\sigma(u)=\gamma(-u)=(-u,0)$, 
%then $\mu^\nu$ changes to 
%$$\mu^\nu=-\dfrac{2\kappa_s(p)}{\sqrt{|\kappa_t'(p)|}}.$$
%\end{rem}

\begin{cor}\label{cor:bdd-zigzag}
Under the same assumptions as in Theorem \ref{cor:sign}, 
if the Gaussian curvature is positive $($\/resp. negative\/$)$ near a cuspidal edge $p$, 
then $p$ is a negative cusp $($\/resp. a positive cusp$\/)$ of the singular locus of the Gauss map.
\end{cor}
\begin{proof}
By Theorems \ref{thm:Gauss-ks} and \ref{cor:sign}, we have the assertion.
\end{proof}

\begin{ex}\label{ex:bounded}
Let $f^\pm\colon\R^2\to\R^3$ be a $C^\infty$ map given by 
$$f^\pm(u,v)=\left(u,\pm3u^2+\dfrac{v^2}{2},\dfrac{v^3}{3}+u^4\pm u^2v^2\right).$$
These maps have cuspidal edge at the origin and it follows that $S(f^\pm)=\{v=0\}$ and $\eta^\pm=\partial_v$.
The Gauss maps $\nu^\pm$ of $f^\pm$ are 
$$\nu^+(u,v)=\dfrac{\left(8u^3-2uv(v-3),-2u^2-v,1\right)}{\sqrt{1+(v+2u^2)^2+(8u^3-2uv(v-3))^2}},$$
and  
$$\nu^-(u,v)=\dfrac{\left(8u^3-2uv(v-3),2u^2-v,1\right)}{\sqrt{1+(v-2u^2)^2+(8u^3-2uv(v-3))^2}},$$ 
respectively. 
In this case, the signed area density function $\lambda^+$ (resp. $\lambda^-$) for $f^+$ (resp. $f^-$) is given as 
$\lambda^+=v\phi$ (resp. $\lambda^-=v\psi$) for some positive function $\phi$ (resp. $\psi$). 
Thus the left-hand side of $\gamma^\pm(u)=(u,0)$ is the region of $\lambda^\pm>0$ near the origin. 

By a direct computation, we have $\kappa_\nu^\pm(u)\equiv0$, 
$$\kappa_s^\pm(u)=\pm\dfrac{6\sqrt{1+24u^4+64u^6}}{(1+36u^2+16u^6)^{3/2}},
\quad\kappa_t^\pm(u)=\pm\dfrac{4u}{1+4u^2+64u^4}$$
along the $u$-axis.
Thus it follows that $\kappa_t^\pm(0)=0$, $\kappa_s^\pm(0)=\pm6$ and $(\kappa_t^\pm)'(0)=\pm4\neq0$. 
This implies that $\nu_\pm$ has a cusp at $(0,0)$, 
and the Gaussian curvature $K^\pm$ of $f^\pm$ is bounded and $K^+<0$ (resp. $K^- >0$) near $(0,0)$ (cf. Theorem \ref{thm:Gauss-ks}). 
In fact, $K^\pm$ are written as 
\begin{align*}
K^+&=\dfrac{-2 \left(3+8 u^2-v\right)}{\left(1+64 u^6+v^2+4u^4 \left(1+24 v-8 v^2\right)+4 u^2 v \left(1+9 v-6 v^2+v^3\right)\right)^2},\\
K^-&=\dfrac{2\left(3-8 u^2-v\right)}{\left(1+64 u^6+v^2+4u^4 \left(1-24 v+8 v^2\right)+4 u^2 v \left(-1+9 v-6 v^2+v^3\right)\right)^2},
\end{align*}
and hence $K^+<0$ and $K^- >0$ on a sufficiently small neighborhood of the origin. 
 
On the other hand, the singular loci $\check{\nu}^\pm(u)=\nu^\pm(u,0)$ is 
$$\check{\nu}^+(u)=\dfrac{\left(8u^3,-2u^2,1\right)}{\sqrt{1+4u^4+64u^6}},\quad 
\check{\nu}^-(u)=\dfrac{\left(8u^3,2u^2,1\right)}{\sqrt{1+4u^4+64u^6}}.$$
These have ordinary cusps at $u=0$, and the cuspidal curvature $\mu_\pm^{\nu}$ at $u=0$ are 
$$\mu_\pm^{\nu}=\pm6=\dfrac{2\kappa_s^\pm(0)}{\sqrt{|(\kappa_t^\pm)'(0)|}}.$$
Thus $(0,0)$ is a positive (resp. negative) cusp of $\check{\nu}^+$ (resp. $\check{\nu}^-$)%, that is, a zig (resp. zag) of $\check{\nu}_+$ (resp. $\check{\nu}_-$)
(cf. Corollary \ref{cor:bdd-zigzag}, see Figures \ref{fig:exa1} and \ref{fig:exa2}).
\begin{figure}[htbp]
  \begin{center}
    \begin{tabular}{c}

      % 1
      \begin{minipage}{0.3\hsize}
        \begin{center}
          \includegraphics[clip, width=3.25cm]{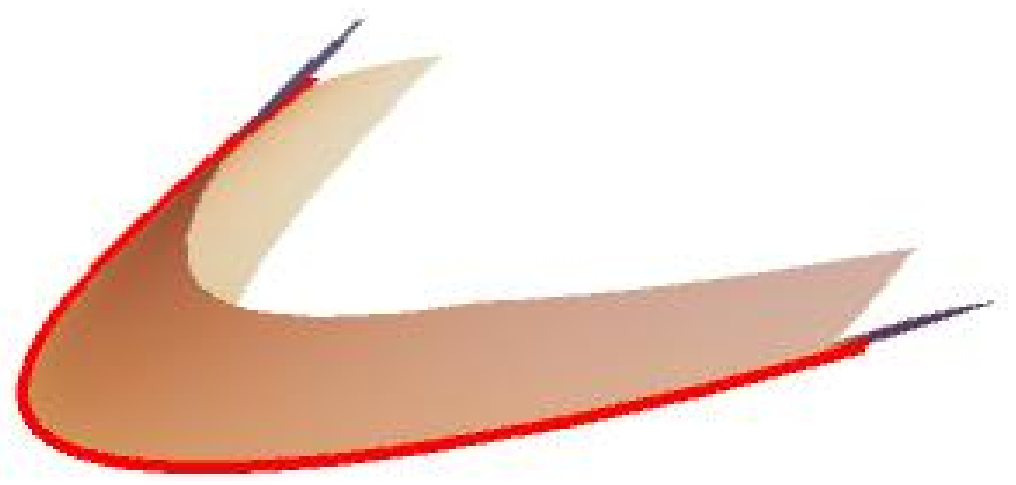}
          %\hspace{1.6cm} lips
        \end{center}
      \end{minipage}

      % 2
      \begin{minipage}{0.3\hsize}
        \begin{center}
          \includegraphics[clip, width=3.25cm]{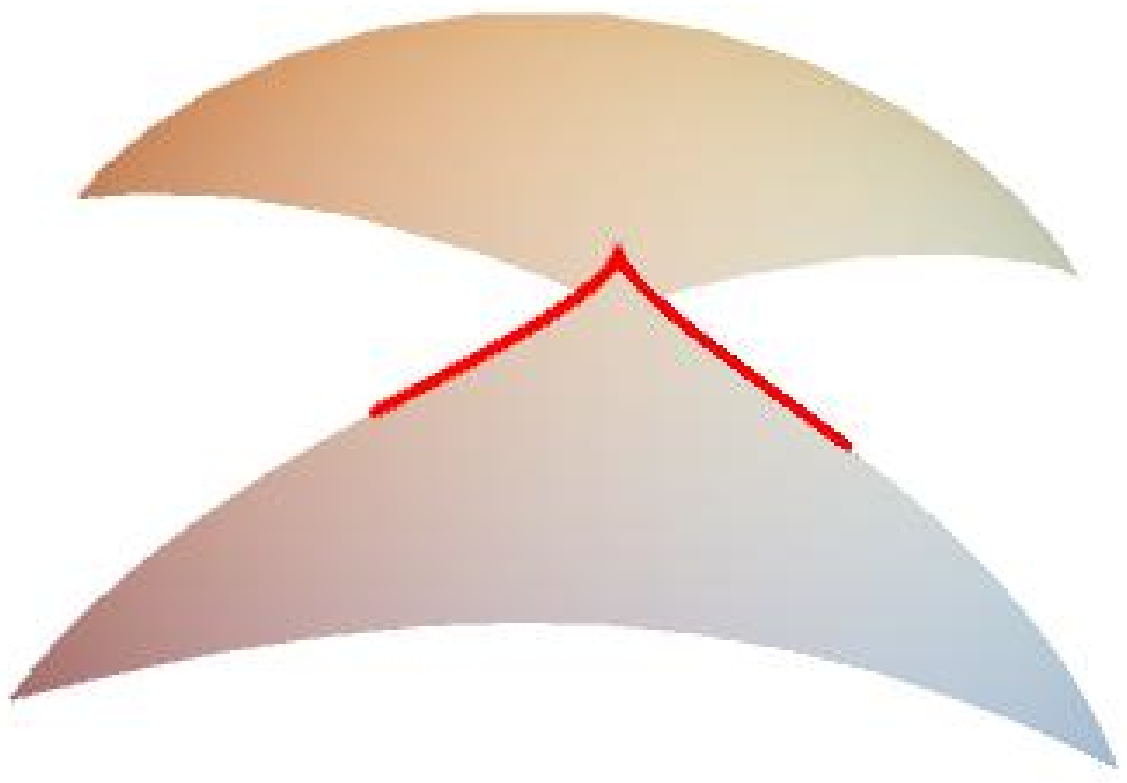}
         % \hspace{1.6cm} beaks
        \end{center}
      \end{minipage}

      % 3
      \begin{minipage}{0.3\hsize}
        \begin{center}
          \includegraphics[clip, width=3.5cm]{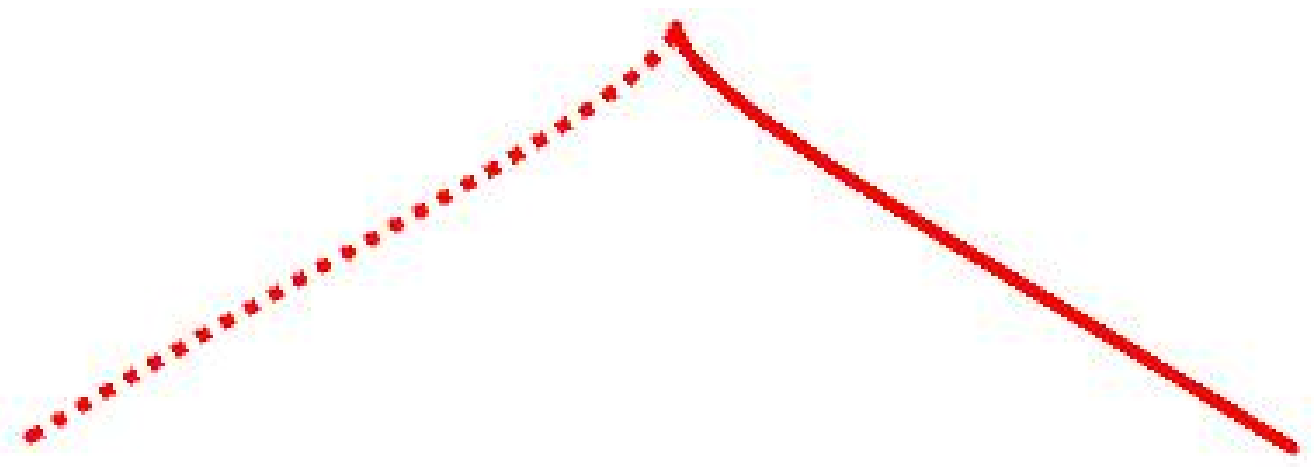}
          %\hspace{1.6cm} swallowtail
        \end{center}
      \end{minipage}

    \end{tabular}
    \caption{Left: The image of $f^+$ given in Example \ref{ex:bounded}.
    The red curve is the image of the curve $\gamma(u)=(u,0)$ by $f^+$. 
    Center: The image of the Gauss map $\nu^+$. 
    Right: The singular locus $\check{\nu}^+$. 
    The dashed curve is the case of $u<0$ and thick curve is of $u>0$. 
    This shows that $\check{\nu}^+$ turns right for the cusp.}
    \label{fig:exa1}
  \end{center}
\end{figure}
\begin{figure}[htbp]
  \begin{center}
    \begin{tabular}{c}

      % 1
      \begin{minipage}{0.3\hsize}
        \begin{center}
          \includegraphics[clip, width=3.25cm]{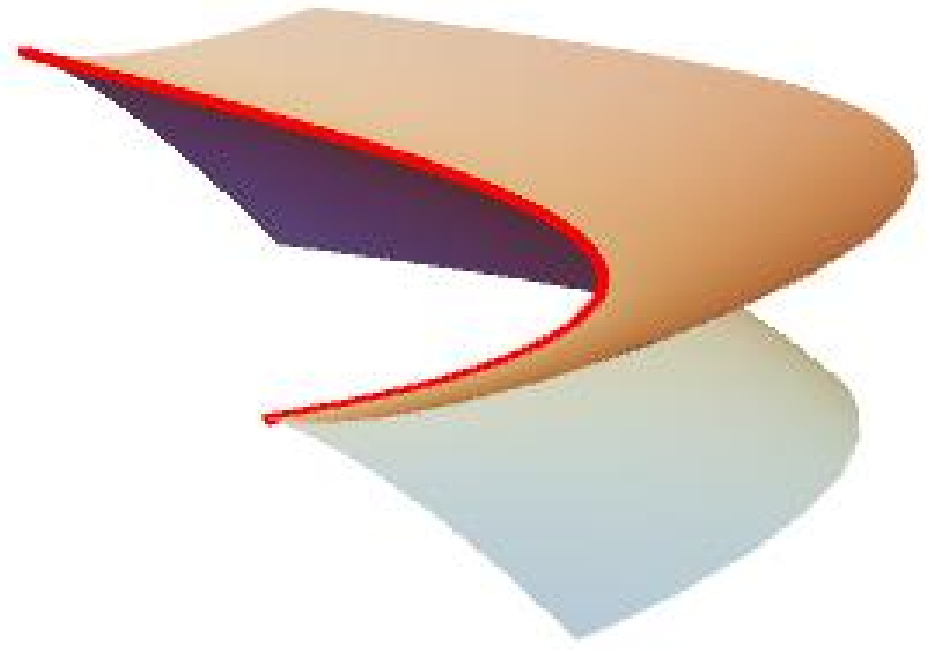}
          %\hspace{1.6cm} lips
        \end{center}
      \end{minipage}

      % 2
      \begin{minipage}{0.3\hsize}
        \begin{center}
          \includegraphics[clip, width=3.25cm]{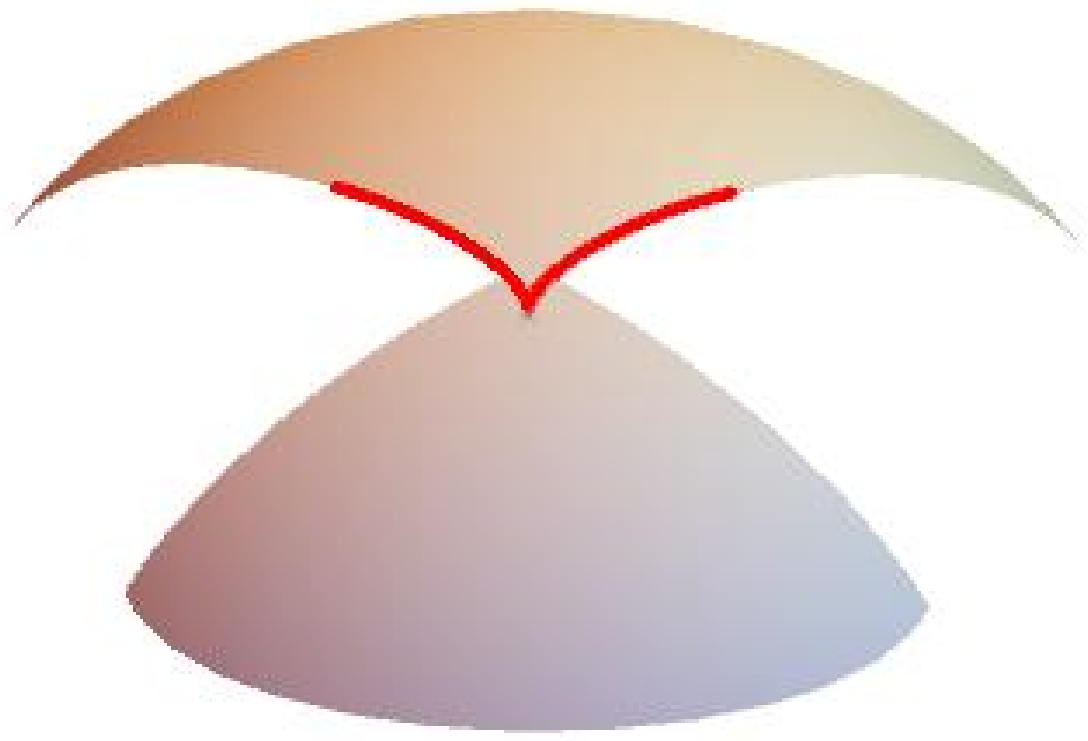}
         % \hspace{1.6cm} beaks
        \end{center}
      \end{minipage}

      % 3
      \begin{minipage}{0.3\hsize}
        \begin{center}
          \includegraphics[clip, width=3.5cm]{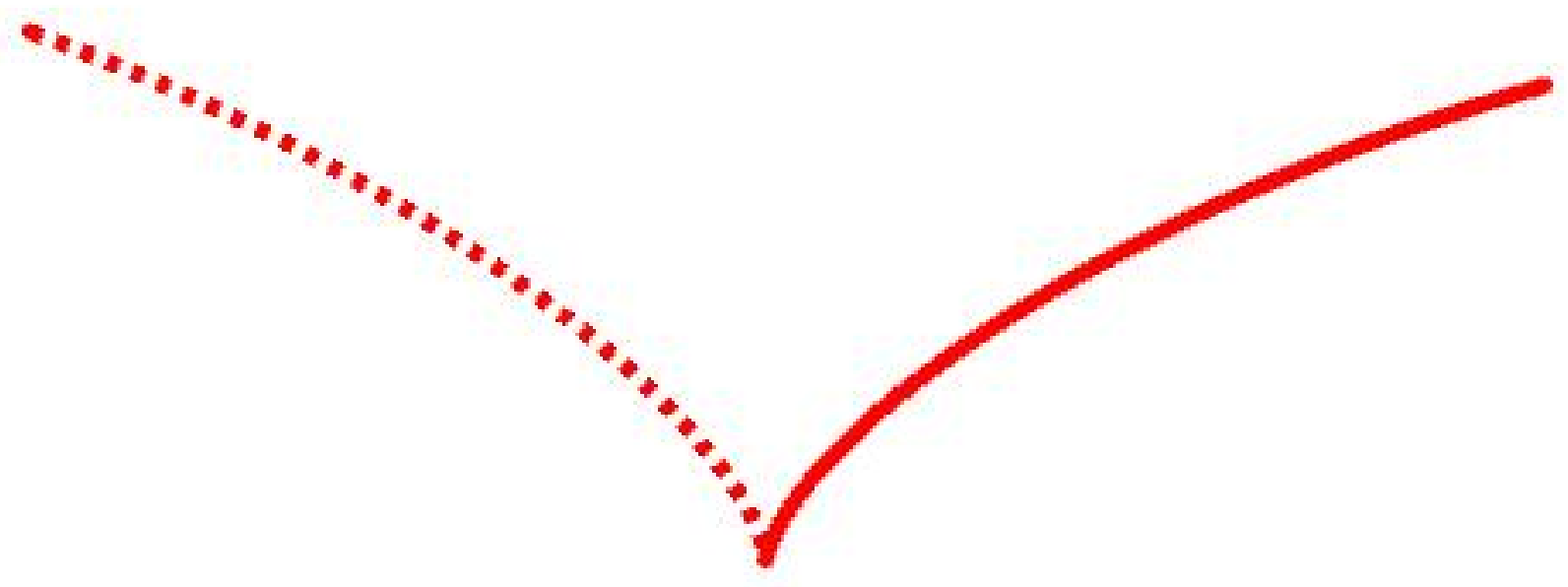}
          %\hspace{1.6cm} swallowtail
        \end{center}
      \end{minipage}

    \end{tabular}
    \caption{Left: The image of $f^-$ given in Example \ref{ex:bounded}.
    The red curve is the image of the curve $\gamma(u)=(u,0)$ by $f^-$. 
    Center: The image of the Gauss map $\nu^-$. 
    Right: The singular locus $\check{\nu}^-$. 
    The dashed curve is the case of $u<0$ and thick curve is of $u>0$. 
    This shows that $\check{\nu}^-$ turns left for the cusp.}
    \label{fig:exa2}
  \end{center}
\end{figure}
\end{ex}

\begin{acknowledgements}
The author is grateful to Professor Kentaro Saji for fruitful discussions and constant encouragements, 
and to Professor Osamu Saeki for valuable comments.
\end{acknowledgements}

%%%%%%%%%%%%%%%%
% bibliography
%%%%%%%%%%%%%%%

%%%%%TEXT END%%%%%

\end{document}